\documentclass[a4paper,11pt]{amsart}
\usepackage[colorlinks, linkcolor=blue,anchorcolor=Periwinkle,
citecolor=blue,urlcolor=Emerald]{hyperref}
\usepackage[all]{xy}
\SelectTips{cm}{}
\usepackage{hyperref}
\usepackage{graphicx}
\usepackage{psfrag}
\usepackage{tikz}
\usepackage{tikz-cd}
\usepackage{mathtools}
\usepackage{amsmath,amssymb,tikz-cd}
\usepackage[export]{adjustbox}
\usepackage{calligra}

\usetikzlibrary{decorations.pathreplacing}
\usetikzlibrary{matrix,arrows}

\usepackage[french,english]{babel}

\numberwithin{equation}{subsection}
\newtheorem{theorem}[subsection]{Theorem}
\newtheorem{definition}[subsection]{Definition}
\newtheorem{classification-theorem}[subsection]{Classification Theorem}
\newtheorem{decomposition-theorem}[subsection]{Decomposition Theorem}
\newtheorem{proposition-definition}[subsection]{Proposition-Definition}
\newtheorem{definition-proposition}[subsection]{Definition-Proposition}
\newtheorem{example-definition}[subsection]{Example-Definition}
\newtheorem{periodicity-conjecture}[subsection]{Periodicity Conjecture}
\newtheorem{lemma}[subsection]{Lemma}
\newtheorem{proposition}[subsection]{Proposition}
\newtheorem{corollary}[subsection]{Corollary}

\newtheorem{example}[subsection]{Example}
\newtheorem{remark}[subsection]{Remark}

\newtheorem{Definition-Proposition}[subsection]{D\'efinition-Proposition}

\newcommand{\reminder}[1]{}

\renewcommand{\mod}{\mathrm{mod}}

\newcommand{\Mod}{\mathrm{Mod}\,}

\newcommand{\proj}{\mathrm{proj}\,}

\newcommand{\add}{\mathrm{add} }

\newcommand{\tr}{\mathrm{tr}}

\newcommand{\pretr}{\mathrm{pretr} }

\newcommand{\cok}{\mathrm{cok} }
\newcommand{\im}{\mathrm{im} }
\renewcommand{\ker}{\mathrm{ker} }

\newcommand{\Q}{\mathbb{Q}}

\newcommand{\iso}{\xrightarrow{_\sim}}

\newcommand{\id}{\mathbf{1}}

\newcommand{\Id}{\mathrm{id}}

\newcommand{\Def}{\mathrm{def}\kern 0.1em}
\newcommand{\gl}{\mathrm{gl. dim}}
\newcommand{\dom}{\mathrm{dom. dim}}

\newcommand{\D}{\mathcal {D}}
\newcommand{\A}{\mathcal {A}}
\newcommand{\B}{\mathcal {B}}
\newcommand{\C}{\mathcal {C}}
\newcommand{\E}{\mathcal {E}}

\newcommand{\J}{\mathcal {J}}
\newcommand{\I}{\mathcal {I}}
\newcommand{\M}{\mathcal {M}}

\newcommand{\U}{\mathcal {U}}

\renewcommand{\P}{\mathcal P}
\newcommand{\X}{\mathcal X}

\newcommand{\Hom}{\mathrm{Hom}}

\newcommand{\Ext}{\mathrm{Ext}}

\newcommand{\End}{\mathrm{End}}

\renewcommand{\phi}{\varphi}

\renewcommand{\tilde}[1]{\widetilde{#1}}

\renewcommand{\Q}{\mathcal{Q}}

\begin{document}
\title[Auslander--Iyama correspondence for exact dg categories]{Auslander--Iyama correspondence for exact dg categories}
\author[Xiaofa Chen] {Xiaofa Chen}
\address{University of Science and Technology of China, Hefei, P.~R.~China}
\email{cxf2011@mail.ustc.edu.cn}

\date{\today}


\begin{abstract} We extend Auslander--Iyama correspondence to the setting of exact dg categories. 
By specializing it to exact dg categories concentrated in degree zero, we obtain a generalization of the higher Auslander correspondence for exact categories due to Ebrahimi--Nasr-Isfahani (in the case of exact categories with split retractions).
\end{abstract}

\maketitle
%\dedicatory{}%
%\commby{}%

\section{Introduction}
Let $R$ be a commutative artinian ring. 
An Artin $R$-algebra is an $R$-algebra which is finitely generated as an $R$-module.
Let $A$ be an Artin $R$-algebra. 
We denote by $\mod A$ the category of finitely generated right $A$-modules. 
Recall that the {\em dominant dimension} of an abelian category $\A$ is the largest $n\in \mathbb N\cup \{\infty\}$ such that any projective object in $\A$ has an injective coresolution whose first $n$ terms are projective. 
The dominant dimension of $A$ is defined as the dominant dimension of $\mod A$. 
An Artin $R$-algebra $A$ is {\em representation-finite} if $\mod A$ has finitely many isomorphism classes of indecomposable objects. 
An Artin $R$-algebra $\Gamma$ is a (1-){\em Auslander algebra} if we have
\[
\gl \Gamma\leq 2\leq \dom \Gamma.
\]
%For a representation finite Artin algebra $A$, the endomorphism algebra $\End_{A}(M)$ for an additive generator $M$ of $\mod A$ is called an {\em Auslander algebra}.
Let us recall Auslander's classical theorem \cite{Auslander71}, which was the starting point and a key tool in developing Auslander--Reiten theory, see for example in \cite[Chapter VI, Proposition 5.11]{AuslanderReitenSmaloe95}.
\begin{theorem}[Auslander correspondence]
There is a bijection between the Morita classes of representation-finite Artin $R$-algebras $A$ and the Morita classes of Auslander $R$-algebras $\Gamma$. 
The bijection is induced by the map $M\mapsto \End_{A}(M)$, where $M$ is an additive generator for $\mod A$.
\end{theorem}
It is natural to ask whether it is possible to replace the number $2$, which appears in the definition of Auslander algebras, by any integer $\geq 2$. 
This was done by Iyama~\cite[Theorem 0.2]{Iyama07}.
Let $d\geq 1$ be any integer. 
Recall that an $A$-module $M$ is {\em $d$-cluster-tilting}~\cite{Iyama07a} if 
\begin{align*}
\add M&=\{X\in\mod A\mid\Ext^{i}(M,X)=0 \text{ for $1\leq i\leq d-1$} \}\\
&=\{Y\in\mod A\mid \Ext^{i}(Y,M)=0 \text{ for $1\leq i\leq d-1$} \}.
\end{align*}
Note that the notion of $1$-cluster-tilting module is equivalent to that of additive generator for $\mod A$. 
An Artin $R$-algebra $\Gamma$ is {\em $d$-Auslander} if we have
\[
\gl \Gamma\leq d+1\leq \dom \Gamma.
\]
Two pairs $(M,A)$ and $(N,B)$, where $A$ and $B$ are Artin $R$-algebras and $M$ (resp.~$N$) is a $d$-cluster-tilting $A$-module (resp.~$B$-module), are {\em equivalent}, if there is an equivalence of $R$-categories $\mod A\iso \mod B$ which restricts to an equivalence $\add M\iso \add N$.
\begin{theorem}[\cite{Iyama07}, Auslander--Iyama correspondence]
There is a bijection between the equivalence classes of pairs (A,M), where $A$ is an Artin $R$-algebras and $M$ is a
 $d$-cluster-tilting module, and the Morita classes of $d$-Auslander $R$-algebras $\Gamma$. 
The bijection is induced by the map $M\mapsto \End_{A}(M)$, where $M$ is $d$-cluster-tilting $A$-module.

\end{theorem}

In this article, we extend Iyama's higher Auslander correspondence to the setting of exact dg categories. 
Let $k$ be a commutative ring. We fix an integer $d\geq 1$.
We refer to \cite{Chen23} for the necessary definitions and properties of exact dg categories and to \cite{NakaokaPalu19} for extriangulated categories.

An extriangulated category with enough projectives is {\em $d$-Auslander} if it has positive global dimension at most $d+1$ and dominant dimension at least $d+1$. It is {\em reduced} if moreover the only projective-injectives are $0$.
An exact dg category $\A$ is {\em $d$-Auslander} if the extriangulated category $H^0(\A)$ is $d$-Auslander.

Let $(\P,\I)$ be a pair where $\P$ is a connective additive dg category and $\I$ is an additive dg subcategory of $\P$.
For $n\geq 0$, we denote by $\A^{(n)}_{\P,\I}$ 
%{\color{red}(we use this notation in order to avoid confusion with the degree n component of the Hom complexes)} 
 the full dg subcategory of $\pretr(\P)$ consisting of the objects in 
\[
\P\ast\Sigma\P\ast\cdots\ast\Sigma^{n-1}\P\ast\Sigma^n\P\cap \ker\Ext^{\geq 1}(-,\I).
\]
The pair $(\P,\I)$ is called {\em $d$-cluster tilting} if $H^0(\I)$ is covariantly finite in $H^0(\A^{(d)}_{\P,\I})$.

\begin{theorem}[=Theorem~\ref{thm:d-Auslander correspondence}]\label{intro:d-Auslander correspondence}
There is a bijective correspondence between the following:
\begin{itemize}
\item[(1)] equivalence classes of connective exact dg categories $\A$ which are $d$-Auslander;
\item[(2)] equivalence classes of $d$-cluster tilting pairs $(\P,\I)$.
\end{itemize}
The bijection from $(1)$ to $(2)$ sends $\A$ to the pair $(\P,\I)$ formed by  the full dg subcategory $\P$ on the projectives of $\A$ and its full dg subcategory $\I$ of projective-injectives. The inverse bijection sends $(\P,\I)$ to the $\tau_{\leq 0}$-truncation of the  dg subcategory $\A^{(d+1)}_{\P,\I}$
 of $\pretr(\P)$.
\end{theorem}
We refer to Definition~\ref{def:equ} for the definition of the equivalence relations mentioned in Theorem~\ref{intro:d-Auslander correspondence}. 
Theorem~\ref{intro:d-Auslander correspondence} also holds when $d=0$. 
In this case, the direction from (2) to (1) was known to Fang--Gorsky--Palu--Plamondon--Pressland, cf.~\cite[Theorem 4.1]{FangGorskyPaluPlamondonPressland23a}.
We show in Example~\ref{exm:Auslander--Iyama correspondence} that  the dg category $\A$ in Theorem~\ref{intro:d-Auslander correspondence}~(1) is the module category over an Artin algebra $\Gamma$ (so in particular $\A$ is concentrated in degree zero) if and only if the corresponding pair $(\P,\I)$ is such that $\P=\add M\subset \mod\Lambda$ for a $d$-cluster tilting module $M$ over some Artin algebra $\Lambda$ and $\I$ is the full subcategory of injective $\Lambda$-modules.

In \cite{HenrardKvammevanRoosmalen22}, the authors introduced the notion of {\em Auslander exact category}.
Ebrahimi and Nasr-Isfahani \cite{EbrahimiNasrIsfahani21} generalized it to the notion of {\em $d$-Auslander exact category}.
Using the notion of $d$-cluster tilting subcategory of an exact category introduced by Jasso \cite[Definition 4,13]{Jasso16}, they proved the following higher Auslander correspondence for exact categories.
\begin{theorem}[{\cite[Theorem 5.2]{EbrahimiNasrIsfahani21}}]\label{intro:EbrahimiNasrIsfahani21}
There is a bijection between the following:
\begin{itemize}
\item[(1)] Equivalence classes of $d$-cluster tilting subcategories of exact categories with enough injectives.
\item[(2)] Equivalence classes of exact categories $\E$ with enough projectives $\P=\proj(\E)$ satisfying the following conditions:
\subitem$\rm{(a)}$ $\gl(\E)\leq d+1\leq \dom(\E)$.
\subitem$\rm{(b)}$ Any morphism $X\rightarrow E$ with $E\in {^{\perp}\P}$ is admissible.
\subitem$\rm{(c)}$ $\P$ is admissibly covariantly finite, cf.~Definition~\ref{def:admissiblycovariantlyfinite}.
\end{itemize}
\end{theorem}
The bijection from (1) to (2) sends $\M$, which is a $d$-cluster tilting subcategory of an exact category $\E$ with enough injectives, to the $d$-Auslander exact category $\mod_{adm}(\M)$, cf.~Definition~\ref{def:admpre}.

If we restrict in Theorem~\ref{intro:d-Auslander correspondence} (1) to the class of dg categories concentrated in degree zero, then we obtain an Auslander--Iyama correspondence for Quillen exact categories. 

\begin{proposition}[{Proposition~\ref{prop:Quillenexact}}]\label{intro:Quillenexact}
The exact dg categories in Theorem~\ref{intro:d-Auslander correspondence} (1) which are Quillen exact correspond to the $d$-cluster tilting pairs $(\P,\I)$ where 
\begin{itemize}
\item[1)]  $\P$ is concentrated in degree $0$;
\item[2)]  for each object $P$ in $\P$, there is a monomorphism $P^{\wedge}\rightarrow I^{\wedge}$ in $\Mod\P$ where $I\in\I$;
\item[3)]  a complex
\begin{equation}\label{introseq:P}
0\rightarrow P_{d+1}^{\wedge}\xrightarrow{f_{d}^{\wedge}} P_{d}^{\wedge}\rightarrow \ldots\xrightarrow{f_{1}^{\wedge}}  P_1^{\wedge}\xrightarrow{f_{0}^{\wedge}} P_0^{\wedge}
\end{equation}
in $\Mod\P$ is exact if the corresponding complex
\begin{equation}\label{introseq:Q}
 \Hom_{\P}(P_0,I)\rightarrow \Hom_{\P}(P_1,I)\rightarrow \ldots\rightarrow\Hom_{\P}(P_{d},I)\rightarrow \Hom_{\P}(P_{d+1},I)\rightarrow 0
\end{equation}
is exact for any $I\in\I$.
\end{itemize}
%By item 1), we identify $\pretr(\P)$ with $\C^b_{dg}(\P)$ and $\A_{\P,\I}^{(d+1)}$ is identified with the full dg subcategory of complexes
%\begin{equation}
%\begin{tikzcd}
%\ldots\ar[r]&0\ar[r]&P_{d+1}\ar[r]&P_{d}\ar[r]&\ldots\ar[r]&P_1\ar[r]&P_0\ar[r]&0\ar[r]&\ldots
%\end{tikzcd}
%\end{equation}
%where $P_i$ is in degree $-i$ and such that the corresponding complex (\ref{introseq:Q})
 %is exact.
 \end{proposition}
%We also introduce the notion of $d$-cluster tilting pair in exact categories, cf.~Definition~\ref{def:d-cluster tilting}.
%Then the pairs satisfying conditions in Proposition~\ref{intro:Quillenexact} correspond to the $d$-cluster tilting pairs in exact categories, cf.~Proposition~\ref{prop:d-cluster tilting bijection1} and Proposition~\ref{prop:d-cluster tilting bijection2}.
\begin{corollary}[Corollary~\ref{cor:Auslander--Iyama correspondence for exact categories}]\label{intro:Auslander--Iyama correspondence for exact categories}%[Corollary~\ref{cor:Auslander--Iyama correspondence exact}]
There is a bijective correspondence between the following:
\begin{itemize}
\item[(1)] equivalence classes of exact categories $\A$ which are $d$-Auslander as extriangulated categories;
\item[(2)] equivalence classes of $d$-cluster tilting pairs $(\P,\I)$ satisfying the conditions in Proposition~\ref{intro:Quillenexact}.
\end{itemize}
The bijection from $(1)$ to $(2)$ sends $\A$ to the pair $(\P, \I)$ where $\P$ is  the full subcategory on the projectives of $\A$ and $\I$ is the full subcategory of projective-injectives. 
The inverse bijection sends $(\P,\I)$ to the $\tau_{\leq 0}$-truncation of the  dg subcategory $\A^{(d+1)}_{\P,\I}$
 of $\pretr(\P)$.
\end{corollary}
Note that in Corollary~\ref{intro:Auslander--Iyama correspondence for exact categories} we consider exact categories to be equivalent if their weakly idempotent completions are equivalent as exact categories, while in~Theorem~\ref{intro:EbrahimiNasrIsfahani21}, the authors consider exact categories to be equivalent if they are equivalent as exact categories.
So we restrict ourselves to the class of weakly idempotent complete exact categories.
Let $\M$ be a $d$-cluster tilting subcategory of a weakly idempotent complete exact category $\E$ with enough injectives.
Let $\I$ be the full subcategory of injectives in $\E$.
Then the pair $(\M,\I)$ is a $d$-cluster tilting pair satisfying the conditions in Proposition~\ref{intro:Quillenexact}, cf.~Proposition~\ref{prop:wicclustertilting}.
We have the following diagram
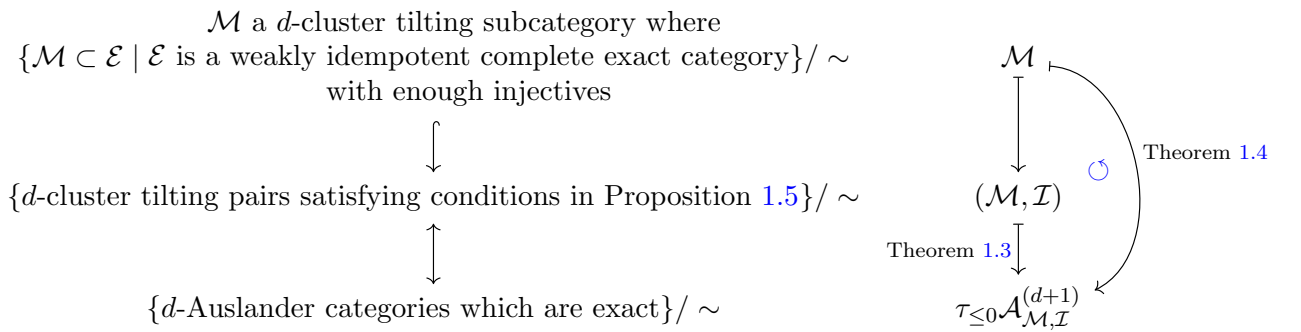
\[
\begin{tikzcd}
\{\M\subset\E \mid \begin{matrix}\text{$\M$ a $d$-cluster tilting subcategory where }\\\text{$\E$ is a weakly idempotent complete exact category}\\ \text{with enough injectives}\end{matrix}\}/\sim\ar[d,hook]&\M\ar[d,mapsto]\ar[dd,mapsto, bend left=16ex,"\text{Theorem~\ref{intro:EbrahimiNasrIsfahani21}}"]\ar[dd,phantom,"\circlearrowleft"blue, bend left=11ex]\\
\{\text{$d$-cluster tilting pairs satisfying conditions in Proposition~\ref{intro:Quillenexact}}\}/\sim\ar[d,<->]&(\M,\I)\ar[d,mapsto,"\text{Theorem~\ref{intro:d-Auslander correspondence}}"swap]\\
\{\text{$d$-Auslander categories which are exact}\}/\sim&\tau_{\leq 0}\A_{\M,\I}^{(d+1)}
\end{tikzcd}
\]
Therefore we obtain a generalization of the higher Auslander correspondence for exact categories in Theorem~\ref{intro:EbrahimiNasrIsfahani21}.

There is another generalization of the class of Auslander algebras, namely the class of Gorenstein algebras $\Gamma$ which satisfy
\begin{equation}\label{ine:Auslander--Gorenstein}
\dom(\Gamma)\geq 2\geq \Id_{\Gamma}(\Gamma)
\end{equation}
where $\Id$ denotes the injective dimension.
The case when both inequalities are equalities was studied in \cite{AuslanderSolberg93d} using the notion of $\tau$-selfinjective algebras.
It was further generalized to the class of {\em $d$-minimal Auslander--Gorenstein algebras} in \cite{IyamaSolberg18} where the number $2$ in (\ref{ine:Auslander--Gorenstein}) is replaced by $d+1$.
Iyama--Solberg proved that there is a bijection between the Morita classes of $d$-minimal Auslander--Gorenstein algebras and equivalence classes of finite {\em $d$-precluster tilting subcategories}.
In \cite{Grevstad22}, the bijection is extended to the setting of exact categories. It would be interesting to study Auslander--Solberg correspondence for exact dg categories.

In this article, additive subcategories of additive categories are assumed to be closed under direct summands.
Throughout we fix an integer $d\geq 1$.
For an additive $k$-category $\P$, we denote by $\Mod \P$ the category of $k$-linear functors $\P^{op}\rightarrow k\mbox{-}\Mod$. For each object $P\in\P$, put $P^{\wedge}=\Hom_{\P}(?,P)\in \Mod\P$.
If $\X$ is a class of objects in $\P$, we write ${^{\perp}\X}\coloneqq\{P\in\P\mid \Hom_{\P}(P,X)=0\text{ for any $X\in\X$ }\}$.

\section{Preliminaries}
\begin{definition}[{\cite[Definition 2.2]{Jasso16}}]\label{def:nkernel}
Let $\C$ be an additive $k$-category and $d^0:X^0\rightarrow X^1$ a morphism in $\C$. An {\em $n$-cokernel} of $d^0$ is a sequence
\begin{equation*}
\begin{tikzcd}
(d^1,\ldots,d^n):X^1\ar[r,"d^1"]& X^2\ar[r,"d^2"]&\cdots\ar[r,"d^n"]&X^{n+1}
\end{tikzcd}
\end{equation*}
such that for all $Y\in\C$ the induced sequence 
\begin{equation*}
\begin{tikzcd}
0\ar[r]&\C(X^{n+1},Y)\ar[r]&\C(X^{n},Y)\ar[r]&\ldots\ar[r]&\C(X^1,Y)\ar[r]&\C(X^0,Y)
\end{tikzcd}
\end{equation*}
is exact. Equivalently, the sequence $(d^1,\ldots,d^n)$ is an $n$-cokernel of $d^0$ if for all $1\leq k\leq n-1$ the morphism $d^{k}$ is a weak cokernel of $d^{k-1}$, and $d^n$ is moreover a cokernel of $d^{n-1}$. The concept of {\em $n$-kernel} of a morphism is defined dually.
\end{definition}
\begin{lemma}[Horseshoe lemma]\label{lem:horseshoe}
Let $(\C,\mathbb E,\mathfrak s)$ be an extriangulated category. Suppose we have the black part in the following diagram in $\C$
\begin{equation}\label{dia:horseshoe}
\begin{tikzcd}
X'\ar[d,tail,"r"swap]\ar[r,red,tail,"{f'}"]&{\color{red} Y'}\ar[d,red,tail,"u"]\ar[r,red,two heads,"{g'}"]&{\color{red}Z'}\ar[d,tail,"s",red]\ar[r,dashed,"\nu"]&\;\\
P\ar[d,two heads,"p"swap]\ar[r,tail,red,"x{=}{[}1{,}0{]}^{\intercal}"]&{\color{red}P\oplus Q}\ar[d,red,two heads,"v"]\ar[r,"y{=}{[}0{,}1{]}",red,two heads]&Q\ar[d,two heads,"q"]\ar[r,dashed,"0"]&\;\\
X\ar[r,tail,"f"swap]\ar[d,dashed,"\delta"]&Y\ar[r,two heads,"g"swap]\ar[d,dashed,"\psi",red]&Z\ar[d,dashed,"\mu",red]\ar[r,dashed,"\tau"swap]&\;\\
\;&\;&\;&
\end{tikzcd}
\end{equation}
where both $P$ and $Q$ are projective in $\C$. 
%Replacing $q:Q\rightarrow Z$ by $(q,0)\colon Q\oplus Q\rightarrow Z$ if necessary (when $\C$ is not weakly idempotent complete), 
We may complete it as above by the red part where $(r,u,s)$, $(p,v,q)$, $(f',x,f)$ and $(g',y,g)$ are morphisms of $\mathbb E$-triangles.
\end{lemma}
\begin{proof}
Since $Q$ is projective in $\C$, we have the following diagram
\[
\begin{tikzcd}
P\ar[d,two heads,"p"swap]\ar[r,tail,red,"x{=}{[}1{,}0{]}^{\intercal}"]&{\color{red}P\oplus Q}\ar[d,red,two heads,"\exists {v'}"]\ar[r,"y{=}{[}0{,}1{]}",red,two heads]&Q\ar[d,equal]\ar[r,dashed,"0"]&\;\\
X\ar[d,equal]\ar[r,tail,red,"{[}1{,}0{]}^{\intercal}"]&{\color{red}X\oplus Q}\ar[d,red,two heads,"\exists {v''}"]\ar[r,"{[}0{,}1{]}",red,two heads]&Q\ar[d,two heads,"q"]\ar[r,dashed,"0"]&\;\\
X\ar[r,tail,"f"swap]&Y\ar[r,two heads,"g"swap]&Z\ar[r,dashed,"\tau"swap]&\;
\end{tikzcd}
\]
By~\cite[Proposition 1.20]{LiuNakaoka19} and~\cite[Lemma A.8]{Chen23}, there exist morphisms $v'$ and $v''$ which are part of morphisms of $\mathbb E$-triangles and which are both $\mathfrak s$-deflations.
Then $v\coloneqq v''v'$ is also an $\mathfrak s$-deflation and $(p,v,q)$ is a morphism of $\mathbb E$-triangles. 
Now we take the idempotent completion $(\tilde{\C},\mathbb F,\mathfrak r)$ of $(\C,\mathbb E,\mathfrak s)$, cf.~\cite{Dixy22}. 
Then we apply~\cite[Lemma 5.9]{NakaokaPalu19} to the diagram
\[
\begin{tikzcd}
X'\ar[d,tail,"r"swap]&{\color{red} Y'}\ar[d,red,tail,"u"] & &\\
P\ar[d,two heads,"p"swap]\ar[r,tail,red,"x{=}{[}1{,}0{]}^{\intercal}"]&{\color{red}P\oplus Q}\ar[d,red,two heads,"v"]\ar[r,"y{=}{[}0{,}1{]}",red,two heads]&Q\ar[r,dashed,"0"]&\;\\
X\ar[r,tail,"f"swap]\ar[d,dashed,"\delta"]&Y\ar[r,two heads,"g"swap]\ar[d,dashed,"\psi",red]&Z\ar[r,dashed,"\tau"swap]&\;\\
\;&\;&\;&
\end{tikzcd}
\]
so that we get the diagram (\ref{dia:horseshoe}) in $\tilde{\C}$ with the required properties. 
It is straightforward to check that the morphism $Q\rightarrow Z$ completing the above diagram is $q$. 
Since $q$ is an $\mathfrak s$-deflation, the object $Z'$ lies in $\C$ and this completes the proof.
%The object $Z'$ need not be in $\C$ a priori. 
%Note that projective objects in $\C$ remain projective in $\tilde{\C}$.
%Since $q$ is an $\mathfrak s$-deflation, we have that $Z'\oplus Q$ lies in $\C$.
%We take the direct sum of the diagram (\ref{dia:horseshoe}) and the diagram
%\[
%\begin{tikzcd}
%0\ar[r]\ar[d]&Q\ar[r,equal]\ar[d,equal]&Q\ar[d,equal]\ar[r,dashed,"0"]&\;\\
%0\ar[r]\ar[d]&Q\ar[r,equal]\ar[d]&Q\ar[r,dashed,"0"]\ar[d]&\;\\
%0\ar[d,dashed,"0"]\ar[r]&0\ar[r]\ar[d,dashed,"0"]&0\ar[d,dashed,"0"]\ar[r,dashed,"0"]&\;\\
%\;&\;&\;&
%\end{tikzcd}
%\]
%and we get the required diagram in $\C$.
\end{proof}
\begin{lemma}\label{lem:wic}
Let $(\C,\mathbb E,\mathfrak s)$ be an extriangulated category.
Let $f\in \C(A,B)$ and $g\in\C(B,C)$ be any composable pair of morphisms. 
If $g\circ f$ is an $\mathfrak s$-deflation and $f$ is an $\mathfrak s$-inflation, then $g$ is an $\mathfrak s$-deflation.
\end{lemma}
\begin{proof}
We take the idempotent completion $(\tilde{\C},\mathbb F,\mathfrak r)$ of $(\C,\mathbb E,\mathfrak s)$, cf.~\cite{Dixy22}. 
Then $g$ is an $\mathfrak r$-deflation and by~\cite[Proposition 3.17]{NakaokaPalu19} we have
\[
\begin{tikzcd}
D\ar[r]\ar[d,red,dashed]&A\ar[r,"gf",two heads]\ar[d,"f",tail]&C\ar[r,dashed,"\delta"]\ar[d,equal]&\;\\
{\color{red}F}\ar[r]\ar[d,dashed,red]&B\ar[d,two heads]\ar[r,two heads,"g"]&C\ar[r,dashed,"\mu"]&\;\\
E\ar[r,equal]\ar[d,dashed,"\nu"]&E\ar[d,dashed,"\psi"]&&\\
\;&\;&&
\end{tikzcd}
\]
The object $F$ is an extension of $D$ and $E$ and hence lies in $\C$. 
Therefore $g$ is an $\mathfrak s$-deflation.
\end{proof}

\begin{definition}[{\cite[Definition 4.13]{Jasso16}}]\label{def:clustertiltingexact}
Let $\E$ be a small exact category and $\M$ a subcategory of $\E$. We say that $\M$ is a {\em $d$-cluster tilting subcategory of} $\E$ if the following conditions are satisfied:
\begin{itemize}
\item[(i)] Every object $E\in\E$ has a left $\M$-approximation by an inflation $E\rightarrowtail M$.
\item[(ii)] Every object $E\in\E$ has a right $\M$-approximation by a deflation $M'\twoheadrightarrow E$.
\item[(iii)] We have 
\begin{align*}
\M&=\{X\in\E \mid \Ext^{i}(\M,X)=0 \text{ for $1\leq i\leq d-1$} \}\\
&=\{Y\in\E \mid \Ext^{i}(Y,\M)=0 \text{ for $1\leq i\leq d-1$} \}.
\end{align*}
\end{itemize}
\end{definition}
A morphism $f:X\rightarrow Y$ in an exact category $\E$ is {\em admissible} if $f$ admits a factorization $X\twoheadrightarrow I\rightarrowtail Y$.
A cochain complex over $\E$
\[
\ldots\rightarrow X^{n-1}\xrightarrow{d^{n-1}} X^n\xrightarrow{d^n} X^{n+1}\rightarrow \ldots
\]
is called {\em acyclic} or {\em exact} if each $d^i$ is admissible and $\ker(d^{i+1})=\im(d^i)$.
\begin{definition}[{\cite[Definition 3.1 (2)]{EbrahimiNasrIsfahani21}}]\label{def:admpre}
Let $\M$ be a full subcategory of an exact category $\E$. 
Let $\mod_{adm}(\M)$ be the full subcategory of $\Mod\M$ consisting of those functors $F$ that admit a projective presentation
\[
X^{\wedge}\xrightarrow{f^{\wedge}} Y^{\wedge}\rightarrow F\rightarrow 0
\]
for some morphism $f:X\rightarrow Y$ in $\M$ which is admissible in $\E$.
\end{definition}

\begin{definition}[{\cite[Definition 4.16]{HenrardKvammevanRoosmalen22}}]\label{def:admissiblycovariantlyfinite}
Let $\M$ be a full subcategory of an exact category $\E$. 
A morphism $f:X\rightarrow M$ in $\E$ with $M\in\M$ is called 
an {\em admissible left $\M$-approximation} if $f$ is an admissible morphism 
and any morphism $X\rightarrow M'$ with $M'\in\M$ factors through $f$. 
The subcategory $\M$ is called {\em admissibly covariantly finite} if 
for all objects $X$ in $\E$ there exists an admissible left $\M$-approximation $X\rightarrow M$.
\end{definition}

\section{Auslander--Iyama correspondence}
\begin{definition}A small extriangulated category $(\C,\mathbb E,\mathfrak s)$ 
has {\em positive global dimension $m$} \cite[Definition 3.28]{GorskyNakaokaPalu21}, 
for some integer $m\geq 0$, if $\mathbb E^m\neq 0$ and $\mathbb E^{m+1}=0$. 
If such an $m$ does not exist, we say that $(\mathcal C,\mathbb E,\mathfrak s)$ has {\em infinite positive global dimension}.
\end{definition}
\begin{definition}
Let $(\C,\mathbb E,\mathfrak s)$ be an extriangulated category with enough projectives. We define its {\em dominant dimension} $\mathrm{dom.dim}(\C,\mathbb E,\mathfrak s)$ \cite[Definition 3.5]{GorskyNakaokaPalu23} to be the largest integer $n$ such that for any projective object $P$, there exist $n$ $\mathfrak s$-triangles
\[
\begin{tikzcd}[row sep=small]
P=Y^0\ar[r,"f^0",tail]&I^0\ar[r,two heads]&Y^1\ar[r,dashed]&;\\
Y^1\ar[r,"f^1",tail]&I^1\ar[r,two heads]&Y^2\ar[r,dashed]&;\\
&\ldots&&\\
Y^{n-1}\ar[r,"f^{n-1}",tail]&I^{n-1}\ar[r,two heads]&Y^{n}\ar[r,dashed]&,
\end{tikzcd}
\]
with $I^k$ being projective-injective for $0\leq k\leq n-1$. 
If such an $n$ doese not exist, we let $\mathrm{dom.dim}(\C, \mathbb E,\mathfrak s)=\infty$.
\end{definition}
\begin{definition}[\cite{GorskyNakaokaPalu23}]
Let $d\geq 1$ be an integer. 
An extriangulated category with enough projectives is {\em $d$-Auslander} if it has positive global dimension at most $d+1$ and has dominant dimension at least $d+1$. It is {\em reduced} if moreover the only projective-injectives are $0$.
\end{definition}
\begin{remark}\label{rmk:selfdual}
The notion of $d$-Auslander extriangulated category is self-dual. The author thanks Yann Palu for pointing out this fact to him.
\end{remark}
\begin{proof}
Let $\C$ be a $d$-Auslander extriangulated category. 
We say that $Y\in \C$ admits an injective coresolution of length $i$ provided that there exists $i$ conflations
\[
\begin{tikzcd}[row sep=small]
Y\ar[r]&J^0\ar[r]&Y^1\ar[r,dashed]&\;\\
Y^1\ar[r]&J^1\ar[r]&Y^2\ar[r,dashed]&\,\;\\
&\ldots&&\\
Y^{i-1}\ar[r]&J^{i-1}\ar[r]&Y^{i}\ar[r,dashed]&\,
\end{tikzcd}
\]
where $J^{k}$ is injective in $\C$ for $0\leq k\leq i-1$. 
%If we have $\mathbb E^{\geq i+1}(-,Y)=0$, then $Y^i$ is injective.
If furthermore in this sequence of conflations $J^{k}$ is injective-projective for $k\leq l-1$, then we say that $Y$ admits a projective-injective coresolution of length $(i,l)$.

We first show that $\C$ has enough injectives.
Let $X$ be an arbitrary object in $\C$.
We have conflations in $\C$
\begin{equation}\label{conf:X}
\begin{tikzcd}[row sep=small]
X_1\ar[r]& P_0\ar[r]& X\ar[r,dashed,"\delta_0"]&\,\\
X_2\ar[r]& P_1\ar[r]& X_1\ar[r,dashed,"\delta_1"]&\,\\
&\ldots&&\\
X_{d+1}=P_{d+1}\ar[r]& P_d\ar[r]& X_d\ar[r,dashed,"\delta_d"]&\,
\end{tikzcd}
\end{equation}
where $P_i$ is projective in $\C$ for $0\leq i\leq d+1$.
We show by induction that for each $i\geq 1$ we have a conflation
\[
\begin{tikzcd}
X_i\ar[r] & I_i\ar[r] &L_i\ar[r,dashed,"\delta'"]&\,
\end{tikzcd}
\]
where $I_i$ is projective-injective and $L_i$ admits a projective-injective coresolution of length $(i,i-1)$.
This clearly holds for $i=d+1$, where $X_{d+1}=P_{d+1}$ is projective, since $\C$ is $d$-Auslander.
Suppose it holds for $i\geq 2$.
We have the following diagram 
\begin{equation}\label{dia:M_i}
\begin{tikzcd}
X_{i}\ar[r,tail]\ar[d,tail]&P_{i-1}\ar[r,two heads,"g_{i-1}"]\ar[d,tail,"v_i"]&X_{i-1}\ar[d,equal]\\
I_i\ar[r,tail]\ar[d,two heads]&M_i\ar[r,two heads]\ar[d,two heads]&X_{i-1}\\
L_{i}\ar[r,equal]&L_{i}&
\end{tikzcd}
\end{equation}
By Lemma~\ref{lem:horseshoe}, the object $M_i$ admits a projective-injective coresolution of length $(i,i-1)$.
So we have the following diagram
\[
\begin{tikzcd}
X_{i-1}\ar[r,equal]\ar[d,tail]&X_{i-1}\ar[d,tail]&\\
M_{i}\ar[r,tail]\ar[d,two heads]&I_{i}'\ar[d,two heads]\ar[r,two heads]&U_{i}\ar[d,equal]\\
I_{i}\ar[r,tail]&V_{i}\ar[r,two heads]&U_{i}
\end{tikzcd}
\]
where $U_i$ admits a projective-injective coresolution of length $(i-1,i-2)$. 
By the Horseshoe Lemma, the object $V_i$ also admits a projective-injective coresolution of length $(i-1,i-2)$.
This finishes the proof of the induction step. When $i=0$, we see that there exists an inflation $X\rightarrowtail J$ where $J$ is injective.

Now we show that $\C$ has codominant dimension at least $d+1$.
Let us assume that $X$ is injective. 
We prove by induction that we have a sequence of conflations~(\ref{conf:X}) where $P_j$ is projective-injective for $0\leq j\leq  d$.
Suppose we have a sequence of conflations~(\ref{conf:X}) where $P_j$ is projective-injective for $j\leq i-2$ for $i\geq 1$.
As above, we have a diagram (\ref{dia:M_i}). 
Then the second row of (\ref{dia:M_i}) splits and we have the following diagram
\[
\begin{tikzcd}
P_{i-1}\oplus I_{i}\ar[r,two heads,"\begin{bmatrix}g_{i-1} \  0\\u_{i} \ \ \ \id\end{bmatrix}"]\ar[d,tail]& X_{i-1}\oplus I_{i}\simeq M_i\\
I_{i-1}\oplus I_{i}\ar[ru,dashed,two heads,"\exists s"swap] &
\end{tikzcd}
\]
where $v_i=[g_{i-1}, u_i]^{\intercal}$ and $s$ is an $\mathfrak s$-deflation by Lemma~\ref{lem:wic}.
So there is an $\mathfrak s$-deflation $I_{i-1}\oplus I_i\rightarrow X_{i-1}$ and this completes the proof of the induction step.
Therefore the codominant dimension is at least $d+1$.
\end{proof}

\begin{definition}
An exact dg category $\A$ is {\em $d$-Auslander} if the extriangulated category $H^0(\A)$ is $d$-Auslander.
\end{definition}

For a connective exact dg category $\A$, we denote by $\overline{\A}$ 
the full dg subcategory of $\D^{b}_{dg}(\A)$ consisting of objects 
in the closure under kernels of retractions of $H^0(\A)$ in $\D^{b}(\A)$. 
Clearly $\overline{\A}$ is stable under extensions in $\D^b_{dg}(\A)$ 
and the inclusion $\A\rightarrow \overline{\A}$ is an exact morphism 
which induces a quasi-equivalence $\D^b_{dg}(\A)\iso \D^b_{dg}(\overline{\A})$.

Let $(\P,\I)$ be a pair where $\P$ is a connective additive dg category and $\I$ is an additive dg subcategory of $\P$.
For $n\geq 0$, we denote by $\A^{(n)}_{\P,\I}$ 
%{\color{red}(we use this notation in order to avoid confusion with the degree n component of the Hom complexes)} 
 the full dg subcategory of $\pretr(\P)$ consisting of the objects in 
\[
\P\ast\Sigma\P\ast\cdots\ast\Sigma^{n-1}\P\ast\Sigma^n\P\cap \ker\Ext^{\geq 1}(-,\I).
\]
We denote by $\Q_{\P,\I}$ (resp.~$\J_{\P,\I}$) the full dg subcategory of $\A_{\P,\I}^{(d+1)}$ whose objects are the direct summands in $H^0(\A_{\P,\I}^{(d+1)})$ of objects in $H^0(\P)\subset H^0(\A_{\P,\I}^{(d+1)})$ (resp.~$H^0(\I)\subset H^0(\A_{\P,\I}^{(d+1)})$).
Clearly the subcateogry $\Q_{\P,\I}$ is stable under kernels of retractions in 
\[
\pretr(\Q_{\P,\I})\iso \pretr(\P).
\]

\begin{definition}\label{def:equ}
\begin{enumerate}
\item[1)] Two connective exact dg categories $\A$ and $\B$ are {\em equivalent} if there is a quasi-equivalence 
%$\D^b_{dg}(\A)\iso \D^b_{dg}(\B)$ inducing a quasi-equivalence 
$\overline{\A}\iso \overline{\B}$.
\item[2)] Let $(\P,\I)$ be a pair consisting of 
 a connective additive dg category $\P$ and
 an additive dg subcategory $\I\subset \P$. It is called {\em $d$-cluster tilting} if $H^0(\I)$ is covariantly finite in $H^0(\A^{(d)}_{\P,\I})$.
\item[3)] Two $d$-cluster tilting pairs $(\P,\I)$ and $(\P',\I')$ are {\em equivalent} if there is a quasi-equivalence 
%\[
%\pretr(\P)\iso \pretr(\P')
%\]
% inducing quasi-equivalences 
 $\Q_{\P,\I}\iso \Q_{\P',\I'}$ which restricts to a quasi-equivalence $ \J_{\P,\I}\iso \J_{\P',\I'}$.
 \end{enumerate}
\end{definition}

\begin{theorem}[Auslander--Iyama correspondence for exact dg categories]\label{thm:d-Auslander correspondence}
There is a bijective correspondence between the following:
\begin{itemize}
\item[(1)] equivalence classes of connective exact dg categories $\A$ which are $d$-Auslander;
\item[(2)] equivalence classes of $d$-cluster tilting pairs.
\end{itemize}
The bijection from $(1)$ to $(2)$ sends $\A$ to the pair $(\P,\I)$ formed by  the full dg subcategory $\P$ on the projectives of $\A$ and its full dg subcategory $\I$ of projective-injectives. The inverse bijection sends $(\P,\I)$ to the $\tau_{\leq 0}$-truncation of the  dg subcategory $\A^{(d+1)}_{\P,\I}$
 of $\pretr(\P)$.
\end{theorem}

Clearly, under the correspondence of Theorem~\ref{thm:d-Auslander correspondence},
\begin{itemize}
\item[a)] the exact dg categories in $(1)$ whose corresponding extriangulated category is moreover reduced correspond to the pairs in $(2)$ where $\I=0$; 
\item[b)] the exact dg categories in $(1)$ with the trivial exact structure (so that each object is both projective and injective) correspond to the pairs in $(2)$ where $\I=\P$.
\end{itemize}
\begin{remark}
Recall that an extriangulated category is {\em algebraic} if it is equivalent as an extriangulated category to $H^0(\A)$ for an exact dg category $\A$, 
cf.~\cite[Proposition-Definition 6.20]{Chen23}. 
Since $H^0(\tau_{\leq 0}\A)$ and $H^0(\A)$ are equal 
as extriangulated categories, we may assume in the above definition that $\A$ is connective.
Therefore by Theorem~\ref{thm:d-Auslander correspondence}, each algebraic $d$-Auslander extriangulated category is of the form $H^0(\A_{\P,\I}^{(d+1)})$ for some $d$-cluster tilting pair $(\P,\I)$.
\end{remark}
\begin{proposition}\label{prop: equivalence d-Auslander}
A connective exact dg category $\A$ is $d$-Auslander if and only if $\overline{\A}$ is $d$-Auslander.
\end{proposition}
\begin{proof}Let $\P$ (resp.~$\I$) be the full dg subcategory on the projectives (resp.~projective-injectives) of $\A$.
 Let $X\in \overline{\A}$. Then we have $Y$ and $Z$ in $\A$ such that $Y\iso X\oplus Z$. We have a conflation in $\A$
\[
M\rightarrow P\rightarrow Y
\]
where $P$ is projective in $H^0(\A)$. Then we have the following diagram in $\D^b(\A)$
\begin{equation}\label{dia:octa}
\begin{tikzcd}
M\ar[r]\ar[d,equal]&N\ar[r]\ar[d]&Z\ar[d]\ar[r]&\Sigma M\ar[d,equal]\\
M\ar[r]&P\ar[r]\ar[d]&Y\ar[d]\ar[r]&\Sigma M\\
&X\ar[r,equal]&X&
\end{tikzcd}
\end{equation}
Then it is clear that $N$ lies in $\A$ and hence $\overline{\A}$ has enough projectives.

Since the dg derived category of $\A$ is quasi-equivalent to the dg derived category $\overline{\A}$ and the higher extensions of $H^0(\A)$ (resp.~$H^0(\overline{\A})$) can be computed as the suspended groups in their dg derived categories, cf.~\cite[Proposition 6.23]{Chen23}, it is immediate that the posive global dimension of $\overline{\A}$ is at most $d+1$.  

 Suppose $X\in \overline{\A}$ is projective. 
 Since $\A$ has enough projectives, we have $Y$ and $Z$ in $\P$ such that $Y\iso X\oplus Z$: in diagram (\ref{dia:octa}), the middle column splits and we have that $N\in \P$ and $P\iso X\oplus N$.
 We have a conflation in $\A$
 \[
 Y\rightarrow I\rightarrow U
 \]
 where $I\in\I$ and $U$ admits a projective-injective coresolution of length $d$.
 So we have the following diagram in $\D^b(\A)$
 \[
 \begin{tikzcd}
 X\ar[r,equal]\ar[d]&X\ar[d]&&\\
 Y\ar[r]\ar[d]&I\ar[r]\ar[d]&U\ar[d,equal]\ar[r]&\Sigma Y\ar[d]\\
 Z\ar[r]&V\ar[r]&U\ar[r]&\Sigma Z
 \end{tikzcd}
 \]
 Then by the Horseshoe Lemma, the object $V$ admits a projective-injective coresolution of length $d$  and the second column shows that $X$ admits a projective-injective coresolution of length $d+1$.
 
\end{proof}

\begin{proposition}\label{prop:weakly idempotent complete}
For a $d$-cluster tilting pair $(\P,\I)$, the category $H^0(\A_{\P,\I}^{(d+1)})$ is stable under kernels of retractions in $\tr(\P)$.
\end{proposition}
\begin{proof}
For $i\geq 1$, let $\U^{(i+1)}$ be the full subcategory of $\tr(\P)$ on the objects in
\[
\P\ast\Sigma\P\ast\cdots\ast\Sigma^{i}\P.
\]
Consider a triangle in $\tr(\P)$
\begin{equation}\label{tri:Ad+1}
\begin{tikzcd}
X\ar[r]&Y\ar[r]&Z\ar[r]&\Sigma X.
\end{tikzcd}
\end{equation}
where $X$ and $Y$ are in $\U^{(i+1)}$.  
Then we have $Z\in \U^{(i+2)}$. 
Indeed, in this case we have the following diagram
\[
\begin{tikzcd}
Y'\ar[r,equal]\ar[d]&Y'\ar[d]&\\
K\ar[r]\ar[d]&P\ar[r]\ar[d,""]&Z\ar[d,equal]\\
X\ar[r]&Y\ar[r]&Z
\end{tikzcd}
\]
where $P\in \P$ and $Y'\in \U^{(i)}\subset \U^{(i+1)}$. 
Since $K$ is an extension of $Y'$ and $X$, it lies in $\U^{(i+1)}$. 
Hence $Z$ lies in $\U^{(i+2)}$.
Consider a split triangle
\[
\begin{tikzcd}
X\ar[r]&Y\ar[r]&Z\ar[r]&\Sigma X
\end{tikzcd}
\]
where $X$ and $Y$ are in $\U^{(d+1)}$.
We have the following diagram
\[
\begin{tikzcd}
Y'\ar[r,equal]\ar[d]&Y'\ar[d]&\\
L\ar[r]\ar[d]&P\ar[r]\ar[d]&X\ar[d,equal]\\
Z\ar[r]&Y\ar[r]&X
\end{tikzcd}
\]
Then we have $L\in \Sigma^{-1}\U^{(d+1)}$.
From the triangle in the first column,  we have that $Z\in \Sigma^{-1}\U^{(d+2)}$.
We have a triangle in $\tr(\P)$
\[
\Sigma^{-1}Q\rightarrow Z\rightarrow W\rightarrow Q
\]
where $Q\in\P$ and $W\in \U^{(d+1)}$.
Since $\Hom(\Sigma^{-1}\P,Z)=0$, we have $Z\in \U^{(d+1)}\ast \Sigma\P\subset \U^{(d+1)}$.

Since $\A^{(d+1)}_{\P,\I}=\U^{(d+1)}\cap \ker\Ext^{\geq 1}(-,\I)$, we have that $H^0(\A^{(d+1)}_{\P,\I})$ is stable under kernels of retractions in $\tr(\P)$.
\end{proof}
%\begin{proof}
%Put $\A=\A_{\P,\I}^{(d+1)}$. Suppose we have $Y\iso X\oplus Z$ in $\tr(\P)$ where both $Y$ and $Z$ are in $\A$.
%We have the diagram~\ref{dia:octa} in $\tr(\P)$ where both $M$ and $P$ are in $\P$. It is enough to show that $N\in\P$.
%We have the triangle~\ref{tri:Z} and the diagram~\ref{dia:octa3} in $\tr(\P)$.
%We have $M\oplus P\iso N\oplus \Sigma^{-1}L$. Since $N$ is an extension of $M$ and $Z$, it lies in $\A$ and in particular it lies in $\pr(\P)$.
%Since $\P$ is weakly idempotent complete, we have that $N$ lies in $\P$.
%\end{proof}

Let $(\P,\I)$ be a pair where $\P$ is a connective additive dg category and $\I$ is an additive dg subcategory of $\P$.
Let $\B^{(n)}_{\P,\I}$ be the full dg subcategory of $\pretr(\P)$ consisting of objects in 
\[
 \I\ast\Sigma\I\ast\cdots\ast\Sigma^{n-1}\I\ast\Sigma^n\P\cap \ker\Ext^{\geq 1}(-,\I).
\]
Let $\C^{(n)}_{\P,\I}$ be the full dg subcategory of $\pretr(\P)$ consisting of the objects in the union of the $\B^{(i)}_{\P,\I}$ for $0\leq i\leq n$.

\begin{lemma}\label{lem:covariantlyfinite}
The pair $(\P,\I)$ is $d$-cluster tilting if and only if $H^0(\I)$ is covariantly finite in $H^0(\C^{(d)}_{\P,\I})$.
\end{lemma}
\begin{proof}
Suppose $H^0(\I)$ is covariantly finite in $H^0(\C^{(d)}_{\P,\I})$.
Let $P$ be an object in $H^0(\P)$. 
By assumption, it admits triangles
\[
\begin{tikzcd}[row sep=small]
P=Y^0\ar[r,"f^0"]&I^0\ar[r]&Y^1\ar[r]&\Sigma P\\
Y^1\ar[r,"f^1"]&I^1\ar[r]&Y^2\ar[r]&\Sigma Y^1\\
&\ldots&&\\
Y^{d-1}\ar[r,"f^{d-1}"]&I^{d-1}\ar[r]&Y^{d}\ar[r]&\Sigma Y^{d-1}
\end{tikzcd}
\]
where $f^{k}$ is a left $H^0(\I)$-approximation for $Y^k$ for $0\leq k\leq d-1$. 
Let $X_1$ be any object in $H^0(\A^{(d)}_{\P,\I})$. 
Then it admits a sequence of triangles (\ref{conf:X}) where $P_i\in H^0(\P)$ for $1\leq i\leq d+1$. 
As in the proof of Remark~\ref{rmk:selfdual}, 
we see that $X_1$ admits a triangle $X_1\xrightarrow{f_1} I_1\rightarrow L_1\rightarrow \Sigma X_1$ 
where $I_1\in H^0(\I)$ and $L_1\in H^0(\A^{(d+1)}_{\P,\I})$. 
It follows that $f_1$ is a left $H^0(\I)$-approximation for $X_1$.
\end{proof}
We start the proof of Theorem~\ref{thm:d-Auslander correspondence} by showing that for a $d$-cluster tilting pair $(\P,\I)$, the dg subcategory $\A_{\P,\I}^{(d+1)}$ of $\pretr(\P)$ inherits a canonical exact dg structure whose corresponding extriangulated category is $d$-Auslander.
\begin{proposition}\label{prop:d-Auslander1}
\begin{itemize}
\item[(1)] For a $d$-cluster tilting pair $(\P,\I)$, the dg subcategory $\A_{\P,\I}^{(d+1)}$ of $\pretr(\P)$ inherits a canonical exact dg structure whose corresponding extriangulated category is $d$-Auslander.
Put $\A=\tau_{\leq 0}\A_{\P,\I}^{(d+1)}$.
\item [(2)] %Let $\Q$ (resp.~$\J$) be the full dg subcategory of $\A$ whose objects are the direct summands in $H^0(\A)$ of objects in $H^0(\P)\subset H^0(\A)$ (resp.~$H^0(\I)\subset H^0(\A)$) (so $H^0(\J)$ is covariantly finite in $H^0(\Q)$). 
We have that $\Q_{\P,\I}$ is the full dg subcategory of projectives in $\A_{\P,\I}^{(d+1)}$ and $\J_{\P,\I}$ is the full dg subcategory of projective-injectives in $\A_{\P,\I}$. Moreover $(\Q_{\P,\I},\J_{\P,\I})$ is a $d$-cluster tilting pair equivalent to $(\P,\I)$.
\item [(3)]The dg derived category of $\A$ is $\pretr(\P)$.
\end{itemize}
\end{proposition}
\begin{proof}
Since $\P$ is connective, the full subcategory $\P\ast\Sigma \P\ast\cdots \ast\Sigma^{d+1}\P$ is extension-closed in $\pretr(\P)$ and hence $\A$ is also extension-closed in $\pretr(\P)$.

It is clear that $\Q_{\P,\I}$ is the full dg subcategory of projective objects: suppose $X$ is projective in $\A$. We have triangles
\begin{equation}\label{tri:X}
\begin{tikzcd}
X_1\ar[r]& P_0\ar[r] &X\ar[r]& \Sigma X_1,
\end{tikzcd}
\end{equation}
\begin{equation}\label{tri:X_1}
\begin{tikzcd}
X_2\ar[r]& P_1\ar[r]& X_1\ar[r] &\Sigma X_2,
\end{tikzcd}
\end{equation}
\[
\ldots
\]
\begin{equation}\label{tri:X_d}
\begin{tikzcd}
X_{d+1}=P_{d+1}\ar[r]&P_{d}\ar[r]&X_{d}\ar[r]&\Sigma X_{d+1}.
\end{tikzcd}
\end{equation}
Since $\Ext^2(X,\I)=0$, from the triangle (\ref{tri:X}) we have $\Ext^1(X_1,\I)=0$ and therefore $X_1$ also lies in $\A$. Similarly, we have that all $X_i$ lie in $\A$ for $1\leq i\leq d+1$. Then since $X$ is projective in $\A$, the first triangle splits and $X_1$ is a direct summand of $P_0$ and hence is projective in $\A$. Similarly, we have that all the triangles split and therefore $X$ lies in $\Q_{\P,\I}$.

It is also clear that the objects in $\I$ are projective-injective.
 
Let us show that $H^0(\A)$ has dominant dimension at least $d+1$: suppose $X$ is projective in $\A$ with the triangles (\ref{tri:X}-\ref{tri:X_d}), then the triangle (\ref{tri:X}) splits. 
By the above, we may assume $X_1\in\P$.
We have a triangle
\begin{equation}\label{tri:I}
P_0\xrightarrow{u} I\xrightarrow{v} Y\xrightarrow{w} \Sigma X
\end{equation}
where $I\in\I$ and $u$ is a left $\I$-approximation of $P_2$. 
Then it is clear that $Y$ lies in $\A$.
We have the following diagram
\begin{equation}
\begin{tikzcd}\label{dia:X}
X\ar[r,equal]\ar[d]&X\ar[d]&&\\
P_0\ar[r,"u"]\ar[d]&I\ar[r,"v"]\ar[d]&Y\ar[d,equal]\ar[r]&\Sigma P_0\ar[d]\\
X_1\ar[r]&U\ar[r]&Y\ar[r]&\Sigma X_1
\end{tikzcd}.
\end{equation}
Since $Y$ admits an $\I$-coresolution of length $d$  and  $X_1\in \P$ has an $\I$-coresolution of length $d+1$, the object $U$ admits an $\I$-coresolution of length $d$ by the Horseshoe Lemma.
%, as the following diagram shows
%\begin{equation}
%\begin{tikzcd}\label{dia}
%\Sigma^{-1}Y\ar[d,"\Sigma^{-1}(g)"swap]\ar[r]&X_1\ar[r]\ar[d,"h"]&U\ar[d]\ar[r]&Y\ar[d,"g"]\\
%\Sigma^{-1}I_0\ar[r,"u"]\ar[d]&I_1\ar[r,"v"]\ar[d]&W\ar[d]\ar[r]&I_0\ar[d]\\
%\Sigma^{-1}J\ar[r]&U_1\ar[r]&Z\ar[r]&J
%\end{tikzcd}
%\end{equation}
%where $g$ is a left $\I$-approximation of $Y$ (so $J\in\A$) and $h$ is a left $\I$-approximation of $X_1$ (so $U_1\in\A$ and $W\in\I$). 
 Hence the middle column of the diagram (\ref{dia:X}) shows that $X$ admits an $\I$-coresolution of length $d+1$.

If $X$ is moreover injective, then the triangle in the middle column of~\ref{dia:X} splits and hence $X$ lies in $\J_{\Q,\I}$.

Let us show that the dg derived category is $\pretr(\P)$. 
By \cite[Proposition 6.23]{Chen23}, it is enough to show that 
the higher extensions of the extriangulated category $H^0(\A)$ 
are naturally isomorphic to the corresponding suspended Hom groups of $\tr(\P)$.
This is clear from the fact that $\P$ is projective in $H^0(\A)$, i.e.~$\Hom_{\tr(\P)}(P,\Sigma^{\geq 1} X)=0$ for $P\in \P$ and $X\in H^0(\A)$, and $\Hom_{\tr(\P)}(X,\Sigma^{i} Y)=0$ for $X$, $Y\in \P$ and $i\geq 1$.
\end{proof}

\begin{proposition}\label{prop:d-Auslander}
Let $\A'$ be a connective exact dg category such that $H^0(\A')$ is $d$-Auslander. 
Suppose that $\P\subset \A'$ is the full dg subcategory of projective objects and $\I\subset \P$ is the full dg subcategory of projective-injective objects (so $(\P,\I)$ is $d$-cluster tilting).

%Let $\I$ be the additive closure of $\J$ in $\Q$ where $\Q$ is the additive closure of $\P$ in $\A$ as described by Lemma~\ref{lem:1-Auslander}. 
Then $\A'$ is equivalent (in the sense of Definition~\ref{def:equ}) to $\A=\tau_{\leq 0}(\A_{\P,\I}^{(d+1)})$.
When $\I=0$, we have that $\A'$ is actually quasi-equivalent to $\A$.
\end{proposition}
\begin{proof}
We have the following diagram
\[
\begin{tikzcd}
\P\ar[r]\ar[d,hook]&\A'\ar[r,hook]&\D^b_{dg}(\A')\\
\pretr(\P)\ar[rru,"\mu"swap]&&
\end{tikzcd}.
\]
The induced morphism $\P\rightarrow \D^b_{dg}(\A')$ is quasi-fully faithful and $\D^b(\A')$ is generated as a triangulated category by the objects in $\P$. 
Therefore the morphism $\mu:\pretr(\P)\rightarrow \D^b_{dg}(\A')$ is a quasi-equivalence.
Recall that $\pretr(\P)$ is the dg derived category of $\A$ by~Proposition~\ref{prop:d-Auslander1} (3).

It is enough to show that the quasi-essential image of $\A\subset \pretr(\P)$ under the morphism $\mu$ is in the closure of $\A'$ under kernels of retractions in $\D^b_{dg}(\A')$.

Let $X\in \D^b(\A')$ be an object with triangles~(\ref{tri:X}), (\ref{tri:X_1}) and (\ref{tri:X_d}). Assume that 
\[
\Hom_{\D^b(\A')}(X,\Sigma^{\geq 1} \I)=0.
\] 
We claim that $X$ lies in the closure of $\A'$ under kernels of retractions in $\D^b(\A')$.
Indeed, since $H^0(\A')$ is $d$-Auslander and the object $X_{d+1}=P_{d+1}$ in the triangle \ref{tri:X_d} is projective in $H^0(\A')$, it admits a triangle as follows
\[
P_{d+1}\rightarrow I_{d+1}\rightarrow U_{d+1}\rightarrow \Sigma P_{d+1}
\]
where $I_{d+1}\in \I$ and $U_{d+1}$ is in $H^0(\A')$.
We have the following diagram
\begin{equation}
\begin{tikzcd}\label{dia:Y}
P_{d+1}\ar[r,""]\ar[d]&P_{d}\ar[r,""]\ar[d]&X_{d}\ar[d,equal]\ar[r]&\Sigma P_{d+1}\ar[d]\\
I_{d+1}\ar[r]\ar[d]&V_{d}\ar[r]\ar[d]&X_{d}\ar[r]&\Sigma I_{d+1}\\
U_{d+1}\ar[r,equal]&U_{d+1}& &
\end{tikzcd}.
\end{equation}
By assumption we have $\Hom(X_d,\Sigma I_{d+1})=0$ and hence $X_{d}\oplus I_{d+1}$ is isomorphic to $V_{d}\in \A'$. 
The object $P_d$ admits an $\I$-coresolution of length $d+1$ and $U_{d+1}$ admits an $\I$-coresolution of length $d$.
From the middle column, we see that $V_d$ admits an $\I$-coresolution of length $d$.
%There exist triangles in $\D^b(\A')$ 
%\[
%P_1\rightarrow I_1\rightarrow U_1\rightarrow \Sigma P_1
%\] 
%and
%\[
%U_2\rightarrow I_2'\rightarrow J_2\rightarrow \Sigma U_2
%\] 
%with $I_1$, $I_2'$ both lie in $\I$ and $U_1$ lies in $\A'$.
%So there is a triangle
%\[
%V\rightarrow I'\rightarrow W\rightarrow \Sigma V
%\]
%with $I'\in \I$ and $W\in \A'$.
So we have the following diagram
\[
\begin{tikzcd}
V_{d}\ar[r]\ar[d]&P_{d-1}\oplus I_{d+1}\ar[r]\ar[d]& X_{d-1}\ar[d,equal]\\
I_{d}\ar[r]\ar[d]&V_{d-1}\ar[d]\ar[r]&X_{d-1}\\
U_{d}\ar[r,equal]&U_{d}&
\end{tikzcd}
\]
where $V_{d-1}$ lies in $\A'$ as an extension of $P_{d-1}\oplus I_{d+1}$ and $U_{d}$, and the second row splits since $\Hom(X,\Sigma I_{d})=0$. We repeat the above procedure and we see that $X$ lies in $\overline{\A'}$.

On the other hand, it is clear that each object in $\A'$ is from an object in $\A$.
\end{proof}

\begin{proof}[Proof of Theorem~\ref{thm:d-Auslander correspondence}]\label{proof:main theorem}
By Proposition~\ref{prop: equivalence d-Auslander}, the property of being $d$-Auslander is invariant under the equivalence relation introduced in Definition~\ref{def:equ}.
Notice that for a $d$-Auslander exact dg category, the inclusion $\A\rightarrow \overline{\A}$ induces a quasi-equivalence $\D^b_{dg}(\A)\iso \D^b_{dg}(\overline{\A})$ and $\D^b_{dg}(\A)$ is quasi-equivalent to the pretriangulated hull of the full dg subcategory of $\A$ on the projectives. 
By the proof of Proposition~\ref{prop: equivalence d-Auslander}, we have that the full dg subcategory $\P_{\overline{\A}}$ of $\overline{\A}$ on the projectives is quasi-equivalent to $\Q_{\P,\I}$. A similar argument shows that the full dg subcategory $\I_{\overline{\A}}$ on the projective-injectives is quasi-equivalent to $\J_{\P,\I}$.
Therefore the map sending the equivalence class of $\A$ to the equivalence class of $(\P,\I)$ is well-defined.

For a $d$-cluster tilting pair $(\P,\I)$, the inclusion dg functor $\P\hookrightarrow\Q_{\P,\I}$ induces a quasi-equivalence $\pretr(\P)\iso \pretr(\Q_{\P,\I})$. 
It is not hard to check that this induces a quasi-equivalence $\A_{\P,\I}^{(d+1)}\iso\A_{\Q_{\P,\I},\J_{\P,\I}}^{(d+1)}$: we observe $\P\ast\Sigma \P=\Q_{\P,\I}\ast\Sigma\Q_{\P,\I}$ and then by induction we have $\Q_{\P,\I}\ast\cdots\ast \Sigma^{n-1}\Q_{\P,\I}\ast\Sigma^n\Q_{\P,\I}=\P\ast\cdots\ast\Sigma^{n-1}\P\ast\Sigma^n\Q_{\P,\I}=\P\ast\cdots\ast\Sigma^{n-1}\P\ast\Sigma^{n}\P$.
Therefore the map sending the equivalence class of a pair $(\P,\I)$ to the equivalence class of $\tau_{\leq 0}\A_{\P,\I}^{(d+1)}$ is well-defined.

Combining Proposition~\ref{prop:d-Auslander1} and Proposition~\ref{prop:d-Auslander}, we see that the above maps are inverse to each other.
\end{proof}
\section{Connection to Quillen exact categories}

In this section we give necessary and sufficient conditions on the $d$-cluster tilting pair $(\P,\I)$ when the connective $d$-Aulander exact dg category $\tau_{\leq 0}\A_{\P,\I}^{(d+1)}$ is concentrated in degree zero (resp.~is an abelian category).
We therefore obtain a generalization of the higher Auslander correspondence for exact categories (resp.~abelian categories) in \cite[Theorem 5.2]{EbrahimiNasrIsfahani21} in the case of exact categories with split retractions (resp.~in \cite[Theorem 8.23]{Beligiannis15}).
\subsection{Exact categories}
\begin{proposition}\label{prop:Quillenexact}
The exact dg categories in Theorem~\ref{thm:d-Auslander correspondence} (1) which are Quillen exact correspond to the $d$-cluster tilting pairs $(\P,\I)$ where 
\begin{itemize}
\item[1)]  $\P$ is concentrated in degree $0$;
\item[2)]  for each object $P$ in $\P$, there is a monomorphism $P^{\wedge}\rightarrow I^{\wedge}$ in $\Mod\P$ where $I\in\I$;
\item[3)]  a complex
\begin{equation}\label{seq:P}
0\rightarrow P_{d+1}^{\wedge}\xrightarrow{f_{d}^{\wedge}} P_{d}^{\wedge}\rightarrow \ldots\xrightarrow{f_{1}^{\wedge}}  P_1^{\wedge}\xrightarrow{f_{0}^{\wedge}} P_0^{\wedge}
\end{equation}
in $\Mod\P$ is exact if the corresponding complex
\begin{equation}\label{seq:Q}
 \Hom_{\P}(P_0,I)\rightarrow \Hom_{\P}(P_1,I)\rightarrow \ldots\rightarrow\Hom_{\P}(P_{d},I)\rightarrow \Hom_{\P}(P_{d+1},I)\rightarrow 0
\end{equation}
is exact for any $I$ in $\I$.
\end{itemize}
For such a pair $(\P,\I)$, if we identify $\pretr(\P)$ with $\C^b_{dg}(\P)$ thanks to item 1), then $\A_{\P,\I}^{(d+1)}$ is identified with the full dg subcategory of complexes
\begin{equation}
\begin{tikzcd}
\ldots\ar[r]&0\ar[r]&P_{d+1}\ar[r]&P_{d}\ar[r]&\ldots\ar[r]&P_1\ar[r]&P_0\ar[r]&0\ar[r]&\ldots
\end{tikzcd}
\end{equation}
where $P_i$ is in degree $-i$ and such that the corresponding complex (\ref{seq:Q})
 is exact for any $I\in\I$.
 \end{proposition}
 
\begin{proof}
Let us assume that $\tau_{\leq 0}\A_{\P,\I}^{(d+1)}$ is concentrated in degree 0. 
Since $\P$ is contained in $\A_{\P,\I}^{(d+1)}$, it is concentrated in degree 0. 
This shows item 1).
Since $\I$ is covariantly finite in $\P$, for each object $P\in \P$, 
there is a left $\I$-approximation $P\rightarrow I$ in $\P$ where $I\in\I$. 
The morphism $P^{\wedge}\rightarrow I^{\wedge}$ is a monomorphism in $\Mod\P$ 
provided $\B^{(1)}_{\P,\I}$ is concentrated in degree 0: we complete it into a triangle
\[
P\rightarrow I\rightarrow M\rightarrow \Sigma P
\]
Then we have $M\in \B^{(1)}_{\P,\I}$. 
For each object $Q\in\P$, we apply the cohomological functor $\Hom(Q,-)$ to the above triangle and we get a long exact sequence
\[
\Hom(Q,\Sigma^{-1}I)\rightarrow \Hom(Q,\Sigma^{-1}M)\rightarrow \Hom(Q,P)\rightarrow\Hom(Q,I)\rightarrow \Hom(Q,M)\rightarrow 0.
\]
The leftmost term is zero since $\P$ is concentrated in degree 0. 
So $\Hom(Q,\Sigma^{-1}M)=0$ for each $Q\in\P$ if and only if $P^{\wedge}\rightarrow I^{\wedge}$ is a monomorphism in $\Mod\P$.
Since $\tau_{\leq 0}\A_{\P,\I}^{(d+1)}$ is concentrated in degree 0, we have $\Hom(\P,\Sigma^{-1}M)=0$ and hence $P^{\wedge}\rightarrow I^{\wedge}$ is a monomorphism in $\Mod\P$. This shows item 2).

If a complex (\ref{seq:P}) in $\Mod\P$ is such that the complex (\ref{seq:Q}) is exact, then it defines an object $X=X_0$ in $\A^{(d+1)}_{\P,\I}$ as follows. First we complete the morphism $f_{d}$ to a triangle:
\[
P_{d+1}\xrightarrow{f_{d}} P_{d}\rightarrow X_{d}\rightarrow \Sigma P_{d+1}.
\]
Then we have $X_{d}\in \A_{\P,\I}^{(d+1)}$. 
Since $\P$ is concentrated in degree $0$, for each $I\in\I$ we have a short exact sequence
\begin{equation}\label{seq:X_d}
0\rightarrow \Hom_{\tr(\P)}(X_{d},I)\rightarrow \Hom_{\tr(\P)}(P_{d},I)\rightarrow \Hom_{\tr(\P)}(P_{d+1},I)\rightarrow 0.
\end{equation}
Since $\P$ is concentrated in degree $0$, 
the morphism $f_{d-1}$ induces a unique map $X_{d}\rightarrow P_{d-1}$ and we complete it into a triangle
\[
X_{d}\rightarrow P_{d-1}\rightarrow X_{d-1}\rightarrow \Sigma X_{d}.
\]
By the sequences (\ref{seq:Q}) and (\ref{seq:X_d}), we have that the morphism 
\[
\Hom(P_{d-1},\I)\rightarrow \Hom(X_d,\I)
\]
 is a surjection and hence $X_{d-1}\in\A_{\P,\I}^{(d+1)}$. We repeat the above procedure and we get a sequence of triangles
\begin{equation}\label{seq:X_i}
X_{i+1}\xrightarrow{f_i} P_{i}\rightarrow X_{i}\rightarrow \Sigma X_{i+1}
\end{equation}
for $0\leq i\leq d$ where $X_{d+1}=P_{d+1}$ and $X_{i}\in \A_{\P,\I}^{(d+1)}$.
Since $\tau_{\leq 0}\A_{\P,\I}^{d+1}$ is concentrated in degree $0$, we have $\Hom_{\tr(\P)}(\Sigma^{j}Q,X_i)=0$ for $Q\in\P$ and $j\geq 1$ and $0\leq i\leq d+1$.
For each $Q\in\P$ we apply the cohomological functor $\Hom(Q,-)$ to the sequences (\ref{seq:X_i}) for $0\leq i\leq d+1$ and we see that the complex (\ref{seq:P}) is exact. 
This shows item 3).

Conversely, suppose the pair $(\P,\I)$ in (2) moreover satisfies the conditions 1)--3). 
Then each object $X=X_0$ in $\A_{\P,\I}^{(d+1)}$ can be obtained by a sequence of triangles (\ref{seq:X_i}) 
for $0\leq i\leq d+1$ where $X_{d+1}=P_{d+1}$ and $P_i\in\P$ and $X_i\in \A_{\P,\I}^{(d+1)}$ for each $i$.
So the complex (\ref{seq:P}) satisfies that the corresponding complex (\ref{seq:Q}) is exact and hence by 3), 
the complex (\ref{seq:P}) is exact. 
We identify $\tr(\P)$ with the homotopy category of bounded complexes over $\P$ and $X$ is identified with the complex
\begin{equation}\label{seq:X}
\ldots\rightarrow 0\rightarrow P_{d+1}\xrightarrow{f_{d}} P_d\xrightarrow{f_{d-1}}\ldots\xrightarrow{f_1} P_1\xrightarrow{f_0} P_0\rightarrow 0\rightarrow \ldots
\end{equation}
where $P_i$ is in degree $i$. By the exactness of the complex (\ref{seq:P}), it is straightforward to verify that $\tau_{\leq 0}\A_{\P,\I}^{(d+1)}$ is concentrated in degree $0$.
\end{proof}
\begin{remark}If a pair $(\P,\I)$ satisfies the conditions in Proposition~\ref{prop:Quillenexact} for an integer $d$, then it does so for all integers $\leq d$.
 \end{remark}
 \begin{corollary}\label{cor:Auslander--Iyama correspondence for exact categories}%[Corollary~\ref{cor:Auslander--Iyama correspondence exact}]
There is a bijective correspondence between the following:
\begin{itemize}
\item[(1)] equivalence classes of exact categories $\A$ which are $d$-Auslander as extriangulated categories;
\item[(2)] equivalence classes of $d$-cluster tilting pairs $(\P,\I)$ satisfying the conditions in Proposition~\ref{intro:Quillenexact}.
\end{itemize}
The bijection from $(1)$ to $(2)$ sends $\A$ to the pair $(\P, \I)$ where $\P$ is  the full subcategory on the projectives of $\A$ and $\I$ is the full subcategory of projective-injectives. 
The inverse bijection sends $(\P,\I)$ to the $\tau_{\leq 0}$-truncation of the  dg subcategory $\A^{(d+1)}_{\P,\I}$
 of $\pretr(\P)$.
\end{corollary}
The following proposition shows that Corollary~\ref{cor:Auslander--Iyama correspondence for exact categories} is a partial generalization of the higher Auslander correspondence for exact categories in Theorem~\ref{intro:EbrahimiNasrIsfahani21}.
 \begin{proposition}\label{prop:wicclustertilting}
 Let $\E$ be a weakly idempotent complete exact category with enough injectives. Let $\I $ be the class of injectives of $\E$.
 Let $\M$ be a $d$-cluster tilting subcategory of $\E$, cf~Definition~\ref{def:clustertiltingexact}. 
 Then the following statements hold.
 \begin{itemize}
 \item[(a)] The pair $(\M,\I)$ is $d$-cluster tilting and satisfies the conditions in Proposition~\ref{prop:Quillenexact}.
 \item[(b)] We have an equivalence of exact categories $H^0(\A_{\M,\I}^{(d+1)})\iso \mod_{adm}(\M)$.
 \end{itemize}
 \end{proposition}
 \begin{proof}
 Clearly $\M$ is concentrated in degree zero and so condition 1) in Proposition~\ref{prop:Quillenexact} holds.
 
 Since $\E$ has enough injectives, for each object $M\in\M$, 
 there is an inflation $M\rightarrowtail I$ where $I\in\I$. 
 The corresponding morphism $M^{\wedge}\rightarrow I^{\wedge}$ is 
 a monomorphism in $\Mod\M$. This shows condition 2).
 
 Consider a complex over $\M$
 \begin{equation}\label{cplx:M}
 0\rightarrow M_{d+1}\xrightarrow{f_d} M_d\xrightarrow{f_{d-1}} \ldots\xrightarrow{f_2} M_2\xrightarrow{f_1} M_1\xrightarrow{f_0} M_0
 \end{equation}
 such that the corresponding complex
\begin{equation}\label{seq:M}
 {\M}(M_0,I)\rightarrow {\M}(M_1,I)\rightarrow {\M}(M_2,I)\rightarrow\ldots\rightarrow {\M}(M_{d},I)\rightarrow {\M}(M_{d+1},I)\rightarrow 0
\end{equation}
is exact for any $I$ in $\I$.
Since $\E$ has enough injectives, there is an inflation $i_{d+1}:M_{d+1}\rightarrowtail I_{d+1}$ where $I_{d+1}\in\I$.
By the exactness of the complex (\ref{seq:M}), there is a morphism $h:M_{d}\rightarrow I_{d+1}$ such that following diagram commutes
\[
\begin{tikzcd}
M_{d+1}\ar[r,"f_{d}"]\ar[d,tail,"i_{d+1}"swap]&M_d\ar[ld,dashed,"h"]\\
I_{d+1}&
\end{tikzcd}
\]
 Since $\E$ is weakly idempotent complete, 
 it follows that $f_d$ is an inflation, cf.~\cite[Proposition 7.6]{Buhler10}.
 Let 
 \[
 \begin{tikzcd}
 M_{d+1}\ar[r,tail,"f_{d}"]&M_{d}\ar[r,two heads,"g_d"] &L_{d}
 \end{tikzcd}
 \]
 be a conflation in $\E$. Then the complex (\ref{cplx:M}) becomes
 \begin{equation}
 \begin{tikzcd}
 M_{d+1}\ar[r,tail,"f_{d}"swap]&M_{d}\ar[rd,two heads,"g_{d}"swap]&&M_{d-1}\ar[r,"f_{d-2}"]&\ldots\ar[r,"f_1"]&M_1\ar[r,"f_0"]&M_0\\
 &&L_d\ar[ru,"v_{d-1}"swap]&
 \end{tikzcd}
 \end{equation}
 Similarly we have that $v_{d-1}:L_d\rightarrow M_{d-1}$ is an inflation. 
 Repeating the above argument, we see that the complex (\ref{cplx:M}) is exact.
 Since $\Ext^{i}(\M,\M)=0$ for $1\leq i\leq d-1$, we see that 
 \begin{equation}\label{seq:MM}
0\rightarrow \M(M,M_{d+1})\rightarrow \M(M,M_d)\rightarrow \ldots\rightarrow \M(M,M_1)\rightarrow \M(M,M_0)
 \end{equation}
 is exact for each $M\in\M$. This shows condition 3).
 
 We identify $\pretr(\M)$ with $\C^b_{dg}(\M)$ thanks to condition 1).
 Let $X$ be an object in 
 \[
 \A_{\M,\I}^{(d)}=\M\ast\Sigma\M\ast\cdots\ast\Sigma^{n-1}\M\ast\Sigma^d\M\cap \ker\Ext^{\geq 1}(-,\I).
 \]
 Then $X$ can be identified with a complex 
  \begin{equation}\label{cplx:M'}
 0\rightarrow M_{d+1}\xrightarrow{f_d} M_d\xrightarrow{f_{d-1}} \ldots\xrightarrow{f_2} M_2\xrightarrow{f_1} M_1
 \end{equation}
  such that the corresponding complex
  \begin{equation}\label{seq:M'}
 {\M}(M_1,I)\rightarrow {\M}(M_2,I)\rightarrow {\M}(M_3,I)\rightarrow\ldots\rightarrow {\M}(M_{d},I)\rightarrow {\M}(M_{d+1},I)\rightarrow 0
\end{equation}
   is exact for any $I\in\I$, where $M_i$ is in degree $-i+1$.
By the above, we see that the complex (\ref{cplx:M'}) is exact and in particular the morphism $f_1$ is admissible.
Let $v_0:M_1\rightarrow L_1$ be its cokernel. 
Since $\E$ has enough injectives, there is an inflation $L_1\rightarrowtail I_0$.
Then the following diagram
\[
\begin{tikzcd}
 0\ar[r] &M_{d+1}\ar[r,"f_d"] &M_d\ar[r,"f_{d-1}"] &\ldots \ar[r,"f_2"]& M_2\ar[r,"f_1"]& M_1\ar[d]\\
           &                               &                            &                         &                        & I_0
 \end{tikzcd}
\]
gives rise to a morphism $X\rightarrow I_0$ in $\mathcal H^b(\M)$ which is a left $\I$-approximation.

The Yoneda embedding $\M\rightarrow \Mod\M$ induces a fully faithful functor $G:\mathcal H^b(\M)\rightarrow \mathcal H^b(\Mod\M)$.
Taking $0$-cohomology gives rise to a functor $H: \mathcal H^b(\Mod\M)\rightarrow \Mod\M$.
Let $F$ be the restriction of the functor $HG$ to the full subcategory $H^0(\A_{\M,\I}^{(d+1)})$ of $\mathcal H^b(\M)$.
Then we obtain a functor $F:H^0(\A_{\M,\I}^{(d+1)})\rightarrow \mod_{adm}(\M)$.
Indeed, let $X$ be an object in $\A_{\M,\I}^{(d+1)}$ which corresponds to the complex (\ref{cplx:M}).
Then $F(X)$ is the cokernel of the morphism $M_1^{\wedge}\xrightarrow{f_0^{\wedge}} M_0^{\wedge}$ in $\Mod\M$.
Since $f_0$ is an admissible morphism, it follows that $F(X)\in \mod_{adm}(\M)$.
Since the complex (\ref{seq:MM}) is exact for each $M\in\M$, we see that the functor $F$ is fully faithful.
%\[
%\begin{tikzcd}
 %0\ar[r] &M_{d+1}^{\wedge}\ar[r,"f_d^{\wedge}"] &M_d^{\wedge}\ar[r,"f_{d-1}^{\wedge}"] &\ldots \ar[r,"f_2^{\wedge}"]& M_2^{\wedge}\ar[r,"f_1^{\wedge}"]& M_1^{\wedge}\ar[r,"f_0^{\wedge}"]&M_0^{\wedge}\ar[r,two heads]&F(X)\ar[r]&0\\
  %0\ar[r] &{M'_{d+1}}^{\wedge}\ar[r,"{f'_d}^{\wedge}"] &{M'_d}^{\wedge}\ar[r,"{f'_{d-1}}^{\wedge}"] &\ldots \ar[r,"{f'_2}^{\wedge}"]& {M'_2}^{\wedge}\ar[r,"{f'_1}^{\wedge}"]& {M'_1}^{\wedge}\ar[r,"{f'_0}^{\wedge}"]&{M'_0}^{\wedge}\ar[r,two heads]&F(X')\ar[r]&0
 %\end{tikzcd}
%\]
For any admissible morphism $f_0:M_1\rightarrow M_0$, we may complete it to an exact sequence in $\E$
\[
\begin{tikzcd}
0\ar[r]& L_{d+1}\ar[r,tail]&M_{d}\ar[r]&\ldots\ar[r]&M_2\ar[r,"f_1"]&M_1\ar[r,"f_0"]&M_0
\end{tikzcd}
\]
where $M_i\in\M$ for $0\leq i\leq d$ and which remains exact after applying the functor $\Hom_{\E}(M,-)$ for each $M\in\M$, by item (ii) in Definition~\ref{def:clustertiltingexact}.
It follows that $\Ext^{i}(M,L_{d+1})=0$ for each $1\leq i\leq d-1$ and $M\in\M$. 
Therefore we have that $L_{d+1}\in\M$ and hence the functor $F$ is dense.

Since $G(X)\in \mathcal H^b(\Mod\M)$ is exact except in degree zero, it follows that $F$ is an exact functor. 
Let 
\[
0\rightarrow F(X)\rightarrow F(Y)\rightarrow F(Z)\rightarrow 0
\]
be a conflation in $\mod_{adm}(\M)$ where $X$, $Y$ and $Z$ are in $H^0(\A_{\M,\I}^{(d+1)})$.
By the Horseshoe Lemma, it can be lifted to a conflation in $\mathcal H^b(\Mod\M)$.
Since $G$ is a fully faithful triangle functor, it then can be lifted to conflation in $H^0(\A_{\M,\I}^{(d+1)})$.
Thereofore the functor $F$ is an equivalence of exact categories.
 \end{proof}

\subsection{Abelian categories}
\begin{proposition}\label{prop:abelian}
The exact dg categories in Theorem~\ref{thm:d-Auslander correspondence} $(1)$ which are abelian categories 
correspond to the $d$-cluster tilting pairs $(\P,\I)$ (assumed to be weakly idempotent complete) in $(2)$ where 
\begin{itemize}
\item[1)]  $\P$ is concentrated in degree $0$ {and has $d$-kernels in the sense of Definition~\ref{def:nkernel}} (in particular the category $\mod\P$ of finitely presented right $\P$-modules is abelian, cf.~\cite{Auslander71});
\item[2)]  for each object $P$ in $\P$, there is a monomorphism $P^{\wedge}\rightarrow I^{\wedge}$ in $\mod\P$ where $I\in\I$;
%\item[3)] {\color{red} $I^{\wedge}$ is injective in $\mod\P$ for $I\in \I$}({\color{red}This is a result of 1), 4) and 5)});
\item[3)]  a complex
\begin{equation}\label{seq:P'}
0\rightarrow P_{d+1}^{\wedge}\xrightarrow{f_{d}^{\wedge}} P_{d}^{\wedge}\rightarrow \ldots\xrightarrow{f_{1}^{\wedge}}  P_1^{\wedge}\xrightarrow{f_{0}^{\wedge}} P_0^{\wedge}
\end{equation}
in $\mod\P$ is exact {iff} the corresponding complex
\begin{equation}\label{seq:Q'}
 \Hom_{\P}(P_0,I)\rightarrow \Hom_{\P}(P_1,I)\rightarrow \ldots\rightarrow\Hom_{\P}(P_{d},I)\rightarrow \Hom_{\P}(P_{d+1},I)\rightarrow 0
\end{equation}
is exact for any $I\in\I$.
%\item[5)] each morphism in $\P$ admits a $d$-kernel in the sense of Definition~\ref{def:nkernel}.
\end{itemize}
In this case, the exact dg category $\tau_{\leq 0}\A_{\P,\I}^{(d+1)}$ is equivalent to the abelian category $\mod\P$ with $I^{\wedge}$, $I\in\I$, the projective-injective objects.
\end{proposition}
\begin{proof}
Suppose that $\tau_{\leq 0}\A_{\P,\I}^{(d+1)}$ is abelian.
Then for each morphism $Q_1\rightarrow Q_2$ in $\P$, there is a kernel $X\in \tau_{\leq 0}\A_{\P,\I}^{(d+1)}$.
Let us assume that $X$ is given by the complex (\ref{seq:X}). Then we have an exact sequence of Hom spaces in $\tr(\P)$
\[
\Hom(P,P_1)\rightarrow \Hom(P,P_0)\rightarrow \Hom(P,X)\rightarrow 0
\]
and hence an exact sequence of Hom spaces in $\P$
\[
\Hom(P,P_0)\rightarrow \Hom(P,Q_1)\rightarrow \Hom(P,Q_2).
\]
This shows that the category $\P$ has weak kernels.

If a morphism $Q_1\rightarrow Q_2$ in $\P$ is such that 
the corresponding morphism in $\mod\P$ is a monomorphism, 
then it is a monomorphism in $\tau_{\leq 0}(\A_{\P,\I}^{(d+1)})$ 
and hence an inflation in the exact dg structure. 
Since the exact dg structure is inherited from $\pretr(\P)$, 
it follows that the mapping cone of the morphism $Q_1\rightarrow Q_2$ is in $\A^{(d+1)}_{\P,\I}$. 
Then we have that the map$\Hom_{\P}(Q_2,I)\rightarrow \Hom_{\P}(Q_1,I)$ is a surjection for each $I\in\I$. 
This shows that $I^{\wedge}$ is injective in $\mod\P$ for $I\in \I$.

For the ``only if " part of item 3), we observe that if the complex (\ref{seq:P'}) is exact, 
then the object given by the complex
\begin{equation}\label{seq:truncatedX}
\ldots\rightarrow 0\rightarrow P_{d+1}\rightarrow P_d\rightarrow\ldots\rightarrow P_{i}\rightarrow 0\rightarrow \ldots,
\end{equation}
where $P_j$ lies in degree $j-i$ for each $j\geq i\geq 1$, 
is an object in $\A_{\P,\I}^{(d+1)}$ by induction on $i$. 
The morphism from the above object to $P_{i-1}$ is a monomorphism in $\tau_{\leq 0}(\A_{\P,\I}^{(d+1)})$. 
Hence we argue as above that its mapping cone also lies in $\tau_{\leq 0}(\A_{\P,\I}^{(d+1)})$, i.e.~the complex~(\ref{seq:truncatedX}) is an object in $\A_{\P,\I}^{(d+1)}$ where $i$ is replaced by $i-1$.
Then by the definition of $\A_{\P,\I}^{(d+1)}$, 
we see that the complex (\ref{seq:Q'}) is exact. 
This shows item 3).

We claim that $\tau_{\leq 0}(\A_{\P,\I}^{(d+1)})$ is equivalent to $\mod\P$ as an exact category.

The map sending a complex (\ref{seq:X}) to the cokernel of 
the map $P_1^{\wedge}\rightarrow P_0^{\wedge}$ in $\mod\P$ 
induces a fully faithful exact functor $\tau_{\leq 0}(\A_{\P,\I}^{(d+1)})\rightarrow \mod\P$.
Since both categories are abelian categories, it is dense and hence an equivalence of categories.
So $\mod\P$ is a $d$-Auslander abelian category.

Conversely, assume that items 1)--3) hold.
By item 1), we identify $\tr(\P)$ with the homotopy category of bounded complexes over $\P$.
Then by defintion the objects in $\A_{\P,\I}^{(d+1)}$ are exactly given by the complexes (\ref{seq:X}) such that the corresponding complex (\ref{seq:Q'}) is exact and hence by item 3) if and only if the complex (\ref{seq:P'}) is exact.
The above functor $\tau_{\leq 0}(\A_{\P,\I}^{(d+1)})\rightarrow \mod\P$ is thus a fully faithful functor.
By item 1), objects in $\mod\P$ have projective dimension at most $d+1$ and hence the above functor is dense.
Therefore $\tau_{\leq 0}(\A_{\P,\I}^{(d+1)})\rightarrow \mod\P$ is an equivalence of categories.
\end{proof}

The following is a direct consequence of Proposition~\ref{prop:wicclustertilting} and Proposition~\ref{prop:abelian}.
\begin{corollary}\label{cor:abelianclustertilting}
Let $\E$ be an abelian category with enough injectives and let $\I$ be the class of injectives of $\E$. 
 Let $\M$ be a $d$-cluster tilting subcategory of $\E$, cf.~Definition~\ref{def:clustertiltingexact}. 
Then the pair $(\M,\I)$ is $d$-cluster tilting and satisfies the conditions in Proposition~\ref{prop:abelian}.
\end{corollary}

It follows from Corollary~\ref{cor:abelianclustertilting} that the following is a generalization of the higher Auslander correspondence for abelian categories \cite[Theorem 8.23]{Beligiannis15}.
\begin{corollary}\label{cor:Auslander--Iyama correspondence abelian}
There is a bijective correspondence between the following:
\begin{itemize}
\item[(1)] equivalence classes of abelian categories $\A$ which are $d$-Auslander as extriangulated categories;
\item[(2)] equivalence classes of $d$-cluster tilting pairs $(\P,\I)$ which satisfy the conditions in Proposition~\ref{prop:abelian}.
\end{itemize}
The bijection from $(1)$ to $(2)$ sends $\A$ to the pair $(\P,\I)$ formed by  the full subcategory $\P$ on the projectives of $\A$ and its full subcategory $\I$ of projective-injectives. 
The inverse bijection sends $(\P,\I)$ to the $\tau_{\leq 0}$-truncation of the  dg subcategory $\A^{(d+1)}_{\P,\I}$
 of $\pretr(\P)$.
\end{corollary}

\begin{example}[Auslander-Iyama correspondence]\label{exm:Auslander--Iyama correspondence}
Let $R$ be a commutative Artin ring. 
Let us fix an $R$-algebra $\Lambda$. 
Let $\I$ be the full subcategory of injective $\Lambda$-modules. 
Let $\P=\add M$ for some $\Lambda$-module $M$. 
We will prove that the $d$-Auslander category $\tau_{\leq 0}\A_{\P,\I}^{(d+1)}$ is equivalent to $\mod \P\iso\mod\Gamma$, where $\Gamma$ is the endomorphism algebra of $M$, if and only if $M$ is a $d$-cluster tilting module.
% in the sense of Iyama~\cite{Iyama07a}. 
In this case, the abelian category $\mod \P$ is $d$-Auslander.

The ``if" part follows from Corollary~\ref{cor:abelianclustertilting}.
Let us show the ``only if" part. 
First, we observe that $\Ext^{i}(M,M)=0$ for $1\leq i\leq d-1$. 
We take an injective coresolution of $M$
\[
0\rightarrow M\rightarrow I_{d}\rightarrow \ldots\rightarrow I_{0}.
\]
By Proposition~\ref{prop:abelian} 3), it remains exact after applying the functor $\Hom_{\Lambda}(M,-)$. 
Therefore $\Ext^{i}(M,M)=0$ for $1\leq i\leq d-1$.

We observe that each $\Lambda$-module $K$ admits a resolution
\begin{equation}\label{seq:K}
0\rightarrow M_{d+1}\rightarrow M_{d}\rightarrow \ldots\rightarrow M_2\rightarrow K\rightarrow 0
\end{equation}
which remains exact when we apply to it the functor $\Hom(M,-)$.  %$\proj\Lambda$ is contained in $\P$.
To see this, let us take an injective copresentation of $K$
\[
0\rightarrow K\xrightarrow{h} I_1\xrightarrow{f_0} I_0
\]
Then we complete it to a projective resolution of $\cok(f_{0}^{\wedge})\in\mod\P$
\[
0\rightarrow M_{d+1}^{\wedge}\rightarrow\ldots\rightarrow M_2^{\wedge}\xrightarrow{f_1^{\wedge}} I_1^{\wedge}\xrightarrow{f_0^{\wedge}} I_0^{\wedge}.
\]
Notice that the functor $\Hom(-,I)$ preserves and reflects exactness for $I$ an injective cogenerator in $\mod\Lambda$. 
Therefore by Proposition~\ref{prop:abelian} 3), the image of the map $f_1:M_2\rightarrow I_1$ is $K$ and we get the required resolution (\ref{seq:K}). 

If we have $\Ext^{i}_{\Lambda}(K,M)=0$ for $1\leq i\leq d-1$, then by applying a dimension shift argument to the sequence (\ref{seq:K}) we see that $K\in \add M$.

Suppose we have $\Ext^{i}_{\Lambda}(M,K)=0$ for $1\leq i\leq d-1$.
Since the subcategory $\I$ of injectives is contained in $\P=\add M$, we have an exact sequence
\begin{equation}\label{seq:ME}
0\rightarrow K\xrightarrow{h} M_{d}\rightarrow M_{d-1}\rightarrow \ldots\rightarrow M_{2}\rightarrow L\rightarrow 0.
\end{equation}
where $M_i\in\P$ for $2\leq i\leq d$ and which remains exact after applying the functor $\Hom(-,M)$.
Then we have $\Ext^j(L,M)=0$ for $1\leq j\leq d-1$ and hence $L\in\add M$.
 Then we apply the functor $\Hom(-,K)$ to the sequence (\ref{seq:ME}) and it remains exact. 
Therefore the morphism $h:K\rightarrow M_d$ splits and $K$ is in $\add M$.
Therefore $M$ is a $d$-cluster tilting module.

Suppose $\Gamma$ is a $d$-Auslander $R$-algebra. 
Let $I$ be an additive generator for the subcategory $\I$ of projective-injectives. 
Put $\Lambda=\End_{\Gamma}(I)$.
Then the functor $D\Hom_{\Gamma}(-,I):\add\Gamma\rightarrow\mod\Lambda$ is fully faithful where $D:\mod\Lambda^{op}\rightarrow \mod\Lambda$ is the duality functor. Indeed, since $\Gamma$ is $d$-Auslander, we have an injective copresentation of $\Gamma$
\[
0\rightarrow \Gamma\rightarrow I_0\rightarrow I_1
\]
which, by the injectivity of $I$, induces an exact sequence of left $\Lambda$-modules
\[
\Hom_{\Gamma}(I_1,I)\rightarrow \Hom_{\Gamma}(I_0,I)\rightarrow \Hom_{\Gamma}(\Gamma,I)\rightarrow 0.
\]
We apply the functor $\Hom_{\Lambda^{op}}(-,\Hom_{\Gamma}(\Gamma,I))$ to the above sequence and we get the following diagram
\[
\begin{tikzcd}[column sep=small]
0\ar[r]& (\Hom_{\Gamma}(\Gamma,I),\Hom_{\Gamma}(\Gamma,I))\ar[r]&(\Hom_{\Gamma}(I_0,I),\Hom_{\Gamma}(\Gamma,I))\ar[r]& (\Hom_{\Gamma}(I_2,I),\Hom_{\Gamma}(\Gamma,I))\\
0\ar[r]&\Hom_{\Gamma}(\Gamma,\Gamma)\ar[r]\ar[u,dashed]&\Hom_{\Gamma}(\Gamma,I_0)\ar[u,"\sim"]\ar[r]&\Hom_{\Gamma}(\Gamma,I_1)\ar[u,"\sim"]
\end{tikzcd}
\]
where the Hom spaces in the first row are in $\mod \Lambda^{op}$.
Therefore the leftmost column is an isomorphism and we see that the functor $D\Hom_{\Gamma}(-,I):\add\Gamma\rightarrow\mod\Lambda$ is fully faithful.
It induces an equivalence between $\I$ and the full subcategory of injective $\Lambda$-modules. By the above, the $\Lambda$-module $D\Hom_{\Gamma}(\Gamma,I)$ is $d$-cluster tilting.
\end{example}

	\def\cprime{$'$} \def\cprime{$'$}
	\providecommand{\bysame}{\leavevmode\hbox to3em{\hrulefill}\thinspace}
	\providecommand{\MR}{\relax\ifhmode\unskip\space\fi MR }
	\providecommand{\MRhref}[2]{%
		\href{http://www.ams.org/mathscinet-getitem?mr=#1}{#2}
	}
	\providecommand{\href}[2]{#2}
	%\begin{thebibliography}{10}
%		

	%\end{thebibliography}

	\bibliographystyle{amsplain}
	\bibliography{stanKeller}

\end{document}